\documentclass[a4paper,12pt,twoside]{article}
\date{November 2, 2016}

\usepackage[english]{babel}
\usepackage{amsfonts}
\usepackage{lmodern,fancyhdr,lastpage,nccfoots,caption}
\usepackage{hyperref}
\usepackage[utf8]{inputenc}
\usepackage{graphicx,color}
\usepackage{amsmath}
\usepackage{amssymb}
\usepackage{amsthm}
\usepackage{enumerate}
\usepackage[a4paper,margin=1in]{geometry}

\theoremstyle{definition} 
\newtheorem{Def}{Definition}[section]

\theoremstyle{plain} 
\newtheorem{Th}[Def]{Theorem}

\newtheorem{Cor}[Def]{Corollary}
\newtheorem{Prop}[Def]{Proposition}

\theoremstyle{remark} 
\newtheorem{Rmk}[Def]{Remark}

\numberwithin{equation}{section}

\newcommand{\suchthat}{\;\;|\;\;}

\renewcommand{\leq}{\leqslant}
\renewcommand{\geq}{\geqslant}

\renewcommand{\phi}{\varphi}

\newcommand{\R}{\mathbb{R}}

\newcommand{\C}{\mathbb{C}}

\newcommand{\dbar}{\bar{\partial}}

\DeclareMathOperator{\HNF}{HNF}
\DeclareMathOperator{\HNT}{HNT}
\DeclareMathOperator{\Gr}{Gr}

\DeclareMathOperator{\rk}{rk}

\DeclareMathOperator{\Hom}{Hom}
\DeclareMathOperator{\End}{End}

\newcounter{a}
\ifodd\thea\else\stepcounter{a}\fi

\fancypagestyle{plain}{%
\fancyhf{}
}
\defineshorthand{"-}{\nobreak\hskip0pt\discretionary{-}{-}{-}\nobreak\hskip0pt}

\begin{document}
%\label{pp}
\thispagestyle{plain}

\begin{center}
\Large
\textsc{Stratifications on the Moduli Space of \\ Higgs Bundles}
\end{center}

\begin{center}
  02/11/2016
\end{center}

\begin{center}
  \textit{Peter B. Gothen\footnote{Partially supported by CMUP
      (UID/MAT/00144/2013) and the projects PTDC/MAT-GEO/0675/2012 and
      PTDC/MAT-GEO/2823/2014, funded by FCT (Portugal) with national
      and, where applicable, European structural funds
      through the programme FEDER, under the partnership agreement PT2020.}}\\
  \small Centro de Matem\'atica da Universidade do Porto CMUP\\
  \small Faculdade de Ci\^encias da Universidade do Porto FCUP\\
  \small Rua do Campo Alegre, s/n\\
  \small 4197-007 Porto, Portugal\\
  \small e-mail: \texttt{pbgothen@fc.up.pt}
\end{center}

\begin{center}
  \textit{Ronald A. Z\'u\~niga-Rojas\footnote{Supported by 
  Universidad de Costa Rica through CIMM (Centro de Investigaciones Matem\'aticas 
  y Metamatem\'aticas), through the Project 820-B5-202. Partially supported by 
  FCT (Portugal) through grant SFRH/BD/51174/2010.}}\\
  \small Centro de Investigaciones Matem\'aticas y Metamatem\'aticas CIMM\\
  \small Universidad de Costa Rica UCR\\
  \small San Jos\'e 11501, Costa Rica\\
  \small e-mail: \texttt{ronald.zunigarojas@ucr.ac.cr}
\end{center}

\noindent
\textbf{Abstract.}  The moduli space of Higgs bundles has two
stratifications. The Bia{\l}ynicki-Birula stratification comes from the
action of the non-zero complex numbers by multiplication on the Higgs
field, and the Shatz stratification arises from the Harder--Narasimhan
type of the vector bundle underlying a Higgs bundle. While these two
stratifications coincide in the case of rank two Higgs bundles, this is
not the case in higher rank. In this paper we analyze the relation
between the two stratifications for the moduli space of rank three Higgs bundles.

\medskip

\noindent\textbf{Keywords:} 
Moduli of Higgs Bundles, Harder--Narasimhan filtrations, Hodge
  Bundles, Vector Bundles.
\noindent\textbf{MSC class:} Primary \texttt{14H60}; Secondary
\texttt{14D07}.	

\section{Introduction}
\label{sec:intro}

Higgs bundles and their moduli were first studied by Hitchin and
Simpson and have been around for almost 30
years. They continue to be the subject of intensive investigations
with links to diverse areas of mathematics such as
non-abelian Hodge theory, integrable systems, mirror symmetry, the
Langlands programme, among others.

In this paper we focus on the moduli space of Higgs bundles on a
compact Riemann surface $X$. The topology of this moduli space has
been studied extensively. Some early calculations of
Betti numbers were carried out by Hitchin \cite{hit2} for rank 2 and
the first author \cite{got} for rank 3. Further significant progress
has been made by a number authors, see, e.g.,
\cite{hath1,hath2,markman:2002,alvarez-consul-garcia-prada-schmitt:2006,
markman:2007,hausel-rodriguez-villegas:2008,decataldo-mark-hausel-migliorini:2012,
hausel-letellier-rodriguez-villegas:2011,hausel:2013,garcia-heinloth:2013,garcia-heinloth-schmitt:2014}.
% on the cohomology ring structure was made by
% Hausel--Thaddeus~\cite{hath1,hath2} (for rank 2) and Markman
% \cite{markman:2002,markman:2007}. Later very significant progress
% has been made by de Cataldo, Hausel, Letellier, Migliorini and
% Rodriguez Villegas in the study of character varieties (see, e.g.,
% \cite{hausel-rodriguez-villegas:2008,decataldo-mark-hausel-migliorini:2012,hausel-letellier-rodriguez-villegas:2011}
% and the survey \cite{hausel:2013}), and Alvar\'ez-Consul,
% Garc\'\i{}a-Prada, Heinloth and Schmitt using holomorphic chains
% (see, e.g.,
% \cite{alvarez-consul-garcia-prada-schmitt:2006,garcia-heinloth-schmitt:2014,garcia-heinloth:2013}).
Recently
Schiffmann \cite{schiffmann:2016} has completely determined the
additive cohomology in the case of Higgs bundles with rank and degree co-prime.

On the other hand, the homotopy theory of the moduli space of Higgs
bundles has not been the subject of a lot of
interest. Hausel~\cite{hau} in his thesis studied the case of rank 2
Higgs bundles, while in \cite{bgg} some results were obtained for
general rank. The latter paper used the Bia{\l}ynicki-Birula
stratification of the Higgs bundle moduli space coming from the
$\C^*$-action given by multiplying the Higgs field by scalars. In rank
2 this stratification coincides with the Shatz stratification, which is given by the
Harder--Narasimhan type of the vector bundle underlying a Higgs
bundle. As already observed by Hitchin and exploited by Hausel and
Thaddeus \cite{hau,hath1} this makes the case of rank 2 Higgs bundles
akin to a finite dimensional version of the infinite dimensional situation
of Atiyah--Bott \cite{atbo}.

However, in general the Bia{\l}ynicki-Birula and Shatz stratifications do not coincide, and it
is therefore of interest to study their relationship. In this paper we
carry out such a study in the case of rank 3 Higgs bundles, where it
turns out that the situation is already fairly complicated. Indeed, our
main result, Theorem~\ref{stratifications}, shows that each
Shatz stratum is intersected by several different
Bia{\l}ynicki-Birula strata. Moreover, knowledge of the underlying vector
bundle of a Higgs bundle is not sufficient to determine its
Bia{\l}ynicki-Birula stratum, one also needs knowledge of the Higgs
field. However, for sufficiently unstable underlying vector bundles
the situation is simpler and the Shatz strata coincide
with Bia{\l}ynicki-Birula strata: this is described in
Theorem~\ref{thm:HN=BB}.

Our results should serve as a useful pointer to the general situation
for higher rank Higgs bundles. Moreover, in the aforementioned work
\cite{hau,hath1}, Hausel and Thaddeus consider the moduli space of
$k$-Higgs bundles (where the Higgs field is allowed to have a pole of
order $k$ at a fixed $p\in X$), and show that in the limit
$k\to\infty$ this moduli space approximates the classifying space of
the gauge group. This is used by Hausel \cite[Theorem~7.5.7]{hau} in
the rank two case to calculate certain homotopy groups of the moduli
space of Higgs bundles, using implicitly that the Bia{\l}ynicki-Birula
and Shatz stratifications coincide. One might thus hope that an
extension of our results to Higgs bundles with poles could be useful
in extending Hausel's results to higher rank.

This paper is organized as follows. In Section~\ref{sec:preliminaries}
we give some preliminaries about Higgs bundles and their moduli spaces
and we explain the Bia{\l}ynicki-Birula and Shatz stratifications of the
moduli space. Next, for completeness, in Section~\ref{sec:rank2} we present the
aforementioned result of Hausel on the equality of the two
stratifications for rank 2 Higgs bundles.  After that, in
Section~\ref{sec:rank-3-higgs}, we give some bounds on the
Harder--Narasimhan types which occur in the moduli space of rank 3
Higgs bundles. Finally, in Section~\ref{sec:limits}, we give our main
results on the relation of the two stratifications.

This paper is partly based on the Ph.D. thesis \cite{z-r2} of the
second author and an announcement of some of our results has appeared in
\cite{z-r1}.

\section{Preliminaries}
\label{sec:preliminaries}

\subsection{Higgs bundles and their moduli}
\label{sec:higgs-bundles}

Let $X$ be a closed Riemann surface of genus $g$ and let
$K=K_X=T^{*}X$ be the canonical line bundle of $X$. 

\begin{Def}
  A \textit{Higgs bundle} over $X$ is a pair $(E, \Phi)$\ where the
  \emph{underlying vector bundle} $E \to X$ is a holomorphic vector
  bundle and the \textit{Higgs field} $\Phi: E \to E\otimes K$ is a
  holomorphic endomorphism of $E$ twisted by $K$.
\end{Def}

The \emph{slope} of a vector bundle $E$ is the quotient between its
degree and its rank: $\mu(E)=\deg(E)/\rk(E)$. Recall that a vector
bundle $E$ is \emph{semistable} if $\mu(F)\leq\mu(E)$ for all non-zero
subbundles $F \subset E$, \emph{stable} if it is semistable and strict
inequality holds for all non-zero proper $F$, and \emph{polystable} if
it is the direct sum of stable bundles, all of the same slope.  
Any semistable vector bundle has a Jordan--H\"older filtration
\begin{math}
  E_0 \subset E_1 \subset \dots \subset E
\end{math}
such that the subquotients $E_j / E_{j-1}$ are stable. The isomorphism
class of the associated graded bundle $\bigoplus E_j / E_{j-1}$ is
unique, and semistable vector bundles are \emph{$S$-equivalent} if
their associated graded bundles are isomorphic. Each $S$-equivalence
class contains a unique polystable representative.  The corresponding
notions for Higgs bundles are defined in exactly the same way, except
that only $\Phi$-invariant subbundles $F \subset E$ (satisfying
$\Phi(F)\subset F\otimes K$) are considered in the stability
conditions.

The moduli space $\mathcal{M}(r,d)$ of $S$-equivalence classes of
semistable rank $r$ and degree $d$ Higgs bundles was constructed by
Nitsure \cite{nit}. The points of $\mathcal{M}(r,d)$ correspond to
isomorphism classes of polystable Higgs bundles. When $r$ and $d$ are
co-prime any semistable Higgs bundle is automatically stable and
$\mathcal{M}(r,d)$ is smooth.  

There are no stable Higgs bundles when $g \leq 1$ and the theory has
quite a different flavour (see, for example, the work of
Franco--Garc\'ia-Prada-Newstead
\cite{franco-et-al:2013,franco-et-al:2014} on Higgs bundles on
elliptic curves), and so we shall also assume that $g \geqslant 2$.

We shall need to consider the moduli space from the complex analytic
point of view. For this, fix a complex $C^\infty$ vector bundle
$\mathcal{E}$ of rank $r$ and degree $d$ on $X$. A holomorphic
structure on $\mathcal{E}$ is given by a $\dbar$-operator
\begin{math}
  \dbar_E\colon A^0(\mathcal{E}) \to {A}^{0,1}(\mathcal{E})
\end{math}
and we thus obtain a holomorphic vector bundle 
$E=(\mathcal{E},\dbar_E)$.  From this point of view, a Higgs bundle
$(E,\Phi)$ arises from a pair $(\dbar_E,\Phi)$ consisting of a
$\dbar$-operator and a Higgs field $\Phi\in
A^{1,0}(\End(\mathcal{E}))$ such that $\dbar_E\Phi=0$. The natural
symmetry group of the situation is the \emph{complex gauge group}
\begin{math}
  \mathcal{G}^\C = \{g\colon \mathcal{E}\to \mathcal{E}\suchthat \text{$g$ is a $C^\infty$ bundle isomorphism}\},
\end{math}
which acts on pairs $(\dbar_E,\Phi)$ in the standard way:
\begin{displaymath}
  g\cdot(\dbar_E,\Phi)=(g\circ\dbar_E\circ g^{-1},g\circ\Phi\circ g^ {-1}).
\end{displaymath}
The moduli space can then be viewed as the quotient\footnote{Strictly speaking one should use appropriate Sobolev 
completions as in {Atiyah and Bott \cite[Section~14]{atbo}}; see, for example, {Hausel and Thaddeus \cite[Section~8]{hath1}}
for the case of Higgs bundles.}
\begin{displaymath}
  \mathcal{M}(r,d) = \{(\dbar_E,\Phi)\suchthat \text{$\dbar_E\Phi=0$ and $(E,\Phi)$ is polystable}\} / \mathcal{G}^\C.
\end{displaymath}

\subsection{Harder--Narasimhan filtrations and the Shatz stratification}
\label{sec:harder-narasimhan}

The Harder--Narasimhan filtration of a vector bundle was introduced in
\cite[Proposition~1.3.9]{hana} and studied systematically by Shatz
\cite[Section~3]{sha}. It plays an important role in the work of
Atiyah and Bott \cite[Section~7]{atbo}. We refer the reader to these
references for details on what follows.

Let $E$ be a holomorphic vector bundle on $X$.  
A \textit{Harder-Narasimhan Filtration} of $E$, is a filtration of the form
\begin{equation}
\label{HN-filtration}
\HNF(E):\ E = E_s \supset E_{s-1} \supset \dots \supset E_1 \supset E_0 = 0
\end{equation}
which satisfies the following two properties:
\begin{enumerate}
 \item[(i)]
 $\mu(E_{j+1}/E_j) < \mu(E_j/E_{j-1})\ \text{for}\ 1\leqslant j \leqslant s-1.$
 \item[(ii)]
 $E_j/E_{j-1}\ \text{is semistable for}\ 1\leqslant j \leqslant s.$
\end{enumerate} 
For brevity, when we have a filtration $E = E_s \supset E_{s-1} \supset \dots \supset E_1 \supset E_0 = 0$ we shall sometimes write 
\begin{math}
  \bar{E}_j = E_j / E_{j-1}
\end{math}
for the subquotients. The associated graded vector bundle is
\begin{displaymath}
  \Gr(E) = \bigoplus_{j=1}^s E_j / E_{j-1} = \bigoplus_{j=1}^s \bar{E}_j.
\end{displaymath}

Any vector bundle $E$ has a unique Harder--Narasimhan filtration. The
subbundle $E_1\subset E$ is called the \emph{maximal destabilizing
  subbundle of $E$}; its rank is maximal among subbundles of $E$ of
maximal slope. Consider
the \emph{Harder--Narasimhan polygon} as the polygon in the $(r,d)$-plane with vertices 
$(\rk(E_j),\deg(E_j))$ for $j=0,\dots,s$. The slope of the line joining $(\rk(E_{j-1}),\deg(E_{j-1}))$ and 
$(\rk(E_j),\deg(E_j))$ is $\mu(\bar{E}_j)$. Condition (i) above says that the Harder--Narasimhan polygon is convex. 
Clearly this is equivalent to saying that $\mu(E_{j}) < \mu(E_{j-1})$ for $j=2,\dots,s$.

The \emph{Harder--Narasimhan type} of $E$ is the following vector in $\R^r$:
\begin{displaymath}
  \HNT(E) = \mu = (\mu(\bar{E_1}),\dots,\mu(\bar{E_1}),\dots,\mu(\bar{E}_s),\dots,\mu(\bar{E}_s))
\end{displaymath}
where $r = \rk(E)$, and the slope of each $\bar{E}_j$ is repeated $r_j = \rk(\bar{E}_j)$ times.

There is a finite decomposition of $\mathcal{M}(r,d)$ by the
Harder--Narasimhan type of the underlying vector bundle $E$ of a Higgs
bundle $(E, \Phi)$:
\begin{equation}
  \label{eq:shatz-stratification}
  \mathcal{M}(r,d) = \bigcup_{\mu}U'_{\mu}
\end{equation}
where $U'_\mu \subset \mathcal{M}(r,d)$\ is the subspace of Higgs
bundles $(E,\Phi)$ whose underlying vector bundle $E$ has
Harder--Narasimhan type $\mu$. When $(E,\Phi)$ is strictly semistable
we take its Harder--Narasimhan type to be that of the polystable
representative of its $S$-equivalence class.  As a consequence of
{Shatz \cite[Propositions~10 and 11]{sha}} the decomposition
(\ref{eq:shatz-stratification}) has nice properties and for this
reason it is known as the \textit{Shatz stratification}. Note that
there is an open dense stratum $U'_{(d/r,\dots,d/r)}$ corresponding to
Higgs bundles $(E,\Phi)$ for which the underlying vector bundle $E$ is
itself semistable (see Hitchin\cite[Proposition~6.1]{hitchin:1987b} in
the rank 2 case and \cite[Proposition~3.12]{bgg} for general
rank). Since $\Phi\in H^0(\End(E)\otimes K)\cong H^1(\End(E))^*$ (by
Serre duality), such a Higgs bundle represents a point in the cotangent
bundle of the moduli space of stable bundles
$\mathcal{N}^s(r,d)$ when $E$ is stable. Thus, if $(r,d)=1$
\begin{displaymath}
%  \label{eq:T-star-ss}
  U'_{(d/r,\dots,d/r)} = T^*\mathcal{N}(r,d)\subset \mathcal{M}(r,d).
\end{displaymath}

\subsection{The $\C^*$-action and the Bia{\l}ynicki-Birula stratification}
\label{sec:c*-action}

We review some standard facts about the $\C^*$-action on
$\mathcal{M}(r,d)$. For more details see, e.g., {Simpson
  \cite[Section~4]{sim}}, especially Lemma~(4.1.).

The holomorphic action of the multiplicative group $\mathbb{C}^{*}$\ on $\mathcal{M}(r,d)$ is defined by the multiplication: 
$$
z\cdot (E,\Phi)\mapsto (E,z\cdot \Phi).
$$

The limit $(E_0,\phi_0)=\lim_{z\to0}z\cdot (E, \Phi)$ 
exists for all $(E,\Phi)\in \mathcal{M}(r,d)$. Moreover, this limit is fixed by the
$\mathbb{C}^{*}$-action. A Higgs bundle $(E,\Phi)$ is a fixed point of the $\C^*$-action 
if and only if it is a \emph{Hodge bundle}, i.e.\ there is a decomposition
\begin{math}
  E=\bigoplus_{j=1}^p E_j
\end{math}
with respect to which the Higgs field has weight one: $\Phi\colon
E_j\to E_{j+1}\otimes K$. The \emph{type} of the Hodge bundle
$(E,\Phi)$ is $(\rk(E_1),\dots,\rk(E_p))$. 

Let $\{F_{\lambda}\}$\ be the irreducible components of the fixed
point locus of $\mathbb{C}^{*}$ on $\mathcal{M}(r,d)$.  Let
$$
U_{\lambda}^+ := \{ (E, \Phi)\in
\mathcal{M}\suchthat\lim_{z\to0}z\cdot (E, \Phi) \in F_{\lambda} \}.
$$
Then we have the \emph{Bia{\l}ynicki-Birula stratification} (cf.\
\cite[Theorem~4.1]{b-b}) of $\mathcal{M}(r,d)$:
$$\mathcal{M} = \bigcup_{\lambda}U_{\lambda}^+.$$
Note that there is a
distinguished component 
$$F_{\min}=\mathcal{N}(r,d)$$ 
of the fixed locus 
corresponding to semistable Higgs bundles with zero Higgs field and
that we have a corresponding Bia{\l}ynicki-Birula stratum $U_{\min}^+$.
Let $(E,\Phi)$ be a semistable Higgs bundle. When the underlying vector
bundle $E$ is itself semistable, clearly 
\begin{math}
  \lim_{z\to 0}z\cdot(E,\Phi)=(E,0).
\end{math}
Conversely, if $\lim_{z\to 0}z\cdot(E,\Phi) =
(E,0)\in\mathcal{M}(r,d)$, then $(E,0)$ is a semistable Higgs bundle
and hence $E$ is a semistable vector bundle. Thus we have the
following result, valid for any rank.

\begin{Prop}
  \label{prop:limit-stable}
  Let $(E,\Phi)\in\mathcal{M}(r,d)$. Then $\lim_{z\to
    0}z\cdot(E,\Phi)=(E,0)$ if and only if $E$ is semistable. \qed
\end{Prop}

In view of this result the following proposition is now immediate.

\begin{Prop}
  The following subspaces of the moduli space $\mathcal{M}(r,d)$
  coincide:
  \begin{displaymath}
    U_{(d/r,\dots,d/r)}' = U_{\min}^+.
  \end{displaymath}
  \qed
\end{Prop}

\section{The rank 2 case}
\label{sec:rank2}

In this section we recall, for completeness, a theorem of Hausel,
which says that in rank $2$ the Shatz and Bia{\l}ynicki-Birula
stratifications coincide.

Let $(E,\Phi)$ be a semistable rank $2$ Higgs bundle
corresponding to a fixed point of the $\C^*$-action on
$\mathcal{M}(2,d)$. In view of the results explained in
Section~\ref{sec:c*-action}, either $\Phi=0$ or $(E,\Phi)$ is of the
form
\begin{equation}
  \label{eq:1}
  (E,\Phi) = (E_1 \oplus E_2, \left( \begin{array}{c c}
                        0 & 0\\ \varphi & 0
                       \end{array}\right)).
\end{equation}
Let $d_1=\deg(E_1)$, then $\deg(E_2)=d-d_1$. Semistability of
$(E,\Phi)$ immediately shows that $d_1$ must satisfy the bounds
\begin{displaymath}
  d \leq 2d_1 \leq d+ 2g-2.
\end{displaymath}
If $d<2d_1$ then $\phi\neq 0$, and if $d=2d_1$ then such a Higgs
bundle is $S$-equivalent to $(E,0)$. Thus, the components of the fixed
locus are $F_{\min}=\mathcal{N}(2,d)$ and, for each $d_1$ with $d < 2d_1 \leq
d+ 2g-2$, a component $F_{d_1}$ consisting of $(E,\Phi)$ of the form
(\ref{eq:1}). (It is easy to see that $F_{d_1}$ is indeed connected, cf.\
Hitchin \cite[Sec.~7]{hit2}.)

The methods employed in the present paper readily give the following
result (cf.\ Remark~\ref{rem:rank2-proof}).

\begin{Prop}
  \label{prop:limit-rank2}
  Let $(E,\Phi)\in \mathcal{M}(2,d)$ be a rank $2$ Higgs bundle such
  that $E$ is an unstable vector bundle with maximal destabilizing
  line bundle $E_1\subset E$. Then the limit $(E_0,\Phi_0)=
  \lim_{z\to0}(E,z\cdot \Phi)$ is
  \begin{displaymath}
       (E_0,\Phi_0) = \Big(E_1\oplus E/E_1,
      \left( \begin{array}{c c}
                            0     & 0\\ 
                        \phi_{21} & 0
                       \end{array}
      \right) \Big),
  \end{displaymath}
  where $\varphi_{21}$ is induced from $\Phi$. The associated graded
  vector bundle is $\Gr(E_0) = \Gr(E)$. \qed
\end{Prop}

Combining Proposition~\ref{prop:limit-rank2} with
Proposition~\ref{prop:limit-stable} immediately gives the following
corollary.

\begin{Cor}[{Hausel \cite[Proposition~4.3.2]{hau}}]\label{[Hau1](4.3.2)}
  The Shatz stratification of $\mathcal{M}(2,d)$ coincides with the
  Bia{\l}ynicki-Birula stratification. More precisely, $U_{(d/2,d/2)}' =
  U_{\min}^+ = T^*\mathcal{N}(2,d)$ (where the last identity holds for
  $d$ odd), and $U'_{d_1,d-d_1} = U_{d_1}^+$
  for each $d_1$ satisfying $d < 2d_1 \leq d+ 2g-2$. \qed
\end{Cor}

\section{Bounds on Harder--Narasimhan types in rank 3}
\label{sec:rank-3-higgs}

Let $(E,\Phi)$ be a rank $3$ Higgs bundle. Let $(\mu_1,\mu_2,\mu_3)$
be the Harder--Narasimhan type of $E$, so that
$\mu_1\geq\mu_2\geq\mu_3$ and $\mu_1+\mu_2+\mu_3=3\mu$, where
$\mu=\mu(E)$. We can write the Harder--Narasimhan filtration of the
vector bundle $E$ as follows:
\begin{displaymath}
  \HNF(E):\ 0 = E_0 \subset E_1 \subset E_2 \subset E_3 = E,
\end{displaymath}
where we have made the convention that $E_i=E_j$ if
$\mu_i=\mu_j$. Thus, for example, if $\mu_1=\mu_2>\mu_3$ then the
Harder--Narasimhan filtration is
\begin{displaymath}
  \HNF(E):\ 0 = E_0 \subset E_1 = E_2 \subset E_3 = E
\end{displaymath}
and $\rk(E_1)=\rk(E_2)=2$. Similarly, if $\mu_1>\mu_2=\mu_3$ then
$\rk(E_1)=1$ and $\rk(E_2)=3$.

We shall next introduce some notation which will be used throughout
the remainder of the paper.

Let $\phi_{21}\colon
  {E}_1 \to E/E_1 \otimes K$  be the map induced by $\Phi$ and
% \begin{displaymath}
%   \phi_{21}\colon E_1 \xrightarrow{\quad \iota \quad} E
%   \xrightarrow{\quad \Phi \quad} E \otimes K
%   \xrightarrow{\pi \otimes 1_K} \big( E/E_1\big) \otimes K.
% \end{displaymath}
let
\begin{equation}
I\subset E/E_1\label{eq:def-I}
\end{equation}
be the subbundle defined by saturating the
subsheaf $\phi_{21}(E_1)\otimes K^{-1} \subset E/E_1$.
Similarly, let $\phi_{32}\colon
  E_2 \to E/E_2 \otimes K$ be the map induced by $\Phi$ and
% \begin{displaymath}
%   \phi_{32}\colon E_2 \xrightarrow{\quad \iota \quad} E
%   \xrightarrow{\quad \Phi \quad} E \otimes K
%   \xrightarrow{\pi \otimes 1_K} \big( E/E_2\big) \otimes K.
% \end{displaymath}
let
\begin{equation}
N =\ker(\phi_{32}) \subset E_2\label{eq:def-N}
\end{equation}
viewed as a subbundle.

\begin{Rmk}
  \label{rmk:phi-non-zero}
  Let $(E,\Phi)$ be a stable Higgs bundle such that $E$ is an unstable
  vector bundle of Harder--Narasimhan type $(\mu_1,\mu_2,\mu_3)$. Then
  $E_1\subset E$ is destabilizing and hence, by 
  stability of $(E,\Phi)$, we have $\varphi_{21}\neq 0$. Similarly 
  $E_2\subset E$ is destabilizing and so $\phi_{32}\neq 0$ (unless
  $\mu_2=\mu_3 \iff E_2=E$).
\end{Rmk}

\begin{Prop}
  \label{prop:mu-inequalities}
  Let $(E,\Phi)$ be a semistable rank $3$ Higgs bundle of
  Harder--Narasimhan type $(\mu_1,\mu_2,\mu_3)$. Then
  \begin{align}
    0 &\leq \mu_1-\mu_2 \leq 2g-2, \label{eq:mu1-mu2}\\
    0 &\leq \mu_2-\mu_3 \leq 2g-2. \label{eq:mu2-mu3}
  \end{align}
\end{Prop}

\begin{proof}
  The fact that the differences $\mu_{i+1}-\mu_i$ are non-negative is
  just the convexity of the Harder--Narasimhan polygon.

  If $E$ is semistable the result is clear, so we may
  assume that this is not the case.

  If $\mu_1>\mu_2$ then $\rk(E_1)=1$, and $I\subset E/E_1$ is a
  line bundle, since $\phi_{21}\neq 0$ by Remark~\ref{rmk:phi-non-zero}. 
  It follows that we have a non-zero map of line bundles
  \begin{math}
    E_1 \to I \otimes K
  \end{math}
  and so
  \begin{displaymath}
    \mu(I) + 2g-2 \geq \mu(E_1)=\mu_1.
  \end{displaymath}
  Also, since $E_2/E_1\subset E/E_1$ is the maximal destabilizing
  subbundle, we have
  that
  \begin{displaymath}
    \mu(I) \leq \mu(E_2/E_1) = \mu_2
  \end{displaymath}
  (note that this inequality also holds if $\mu_2=\mu_3$).
  Combining these two inequalities proves (\ref{eq:mu1-mu2}).

  If $\mu_2>\mu_3$ then $\rk(E_2)=2$, and $N \subset E_2$ is a line
  bundle, since $\phi_{32}\neq 0$ by Remark~\ref{rmk:phi-non-zero}.  
  It follows that we have a non-zero map of line bundles
  \begin{math}
    E_2/N \to E/E_2 \otimes K
  \end{math}
  and so
  \begin{align*}
    \mu(E/E_2) + 2g-2 &\geq \mu(E_2/N) \\
    \iff \mu_3 +2g-2 &\geq \deg(E_2)-\mu(N)=\mu_1+\mu_2-\mu(N).
  \end{align*}
  Also, since $E_1\subset E_2$ is maximal destabilizing, we have
  that
  \begin{displaymath}
    \mu(N) \leq \mu(E_1) = \mu_1
  \end{displaymath}
  (note that this inequality also holds if $\mu_1=\mu_2$).
  Combining these two inequalities proves (\ref{eq:mu2-mu3}).
\end{proof}

Note that the proof of the preceding Proposition gives the following
bounds on the slopes of the bundles $I$ and $N$.

\begin{Prop}
\label{prop:mu-I-N}  Let $(E,\Phi)$ be a semistable rank $3$ Higgs bundle of
  Harder--Narasimhan type $(\mu_1,\mu_2,\mu_3)$ and define $I\subset
  E/E_1$ and $N \subset E_2$ as above.
  \begin{itemize}
  \item[(1)] If $\mu_1>\mu_2$ then $I\subset E/E_1$ is a line
    subbundle of a rank $2$ bundle and $\mu_1-(2g-2) \leq \mu(I) \leq
    \mu_2$.
  \item[(2)] If $\mu_2>\mu_3$ then $N\subset E_2$ is a line subbundle
    of a rank $2$ bundle and $\mu_1+\mu_2-\mu_3-(2g-2) \leq \mu(N)
    \leq \mu_1$.
  \end{itemize}
\qed
\end{Prop}

\section{Limits of the $\C^*$-action}
\label{sec:limits}

The purpose of the present section is to analyse the limit as $z\to 0$
of $z\cdot(E,\Phi)$ as a function of the Harder--Narasimhan type of
$E$. Note that the case of trivial Harder--Narasimhan filtration,
corresponding to $(E,\Phi)$ with semistable underlying vector bundle
$E$, is covered by Proposition~\ref{prop:limit-stable}.

\subsection{Non-trivial Harder--Narasimhan filtrations}

Again we limit ourselves to considering rank $3$
stable Higgs bundles $(E,\Phi)$. We shall use the notation
introduced in Section~\ref{sec:rank-3-higgs}.

\begin{Th}\label{stratifications}
  Let $(E,\Phi) \in \mathcal{M}(3,d)$ be such that $E$ is an unstable
  vector bundle of slope $\mu$ and with Harder--Narasimhan type
  $(\mu_1,\mu_2,\mu_3)$. Then the limit $(E_0,\Phi_0)=
  \lim_{z\to0}(E,z\cdot \Phi)$ is given as follows.

  \begin{enumerate}
  \item[(1)] Assume that $\mu_2<\mu$. Then $\mu_1>\mu_2\geq \mu_3$,
    the subbundle $I\subset E/E_1$ defined in \eqref{eq:def-I} is a
    line bundle, 
    and one of the following alternatives holds.
    \begin{enumerate}
    \item[(1.1)] The slope of $I$ satisfies $\mu_1-(2g-2) \leq
      \mu(I)<-\frac{1}{3}\mu_1+\frac{2}{3}\mu_2+\frac{2}{3}\mu_3$. Then
      $(E_0,\Phi_0)$ is the following Hodge bundle of type $(1,2)$:
      $$
      (E_0,\Phi_0) = \Big(E_1\oplus E/E_1,
      \left( \begin{array}{c c}
                            0     & 0\\ 
                        \phi_{21} & 0
                       \end{array}
      \right) \Big),
      $$
      where $\varphi_{21}$ is induced from $\Phi$.
      The associated graded vector bundle is $\Gr(E_0) = \Gr(E)$.

   \item[(1.2)] The slope of $I$ satisfies $\mu(I) =
     -\frac{1}{3}\mu_1+\frac{2}{3}\mu_2+\frac{2}{3}\mu_3$. Then
     $(E_0,\Phi_0)$ is the following strictly polystable Hodge bundle:
      $$(E_0,\Phi_0) = \Big(E_1\oplus I, \left( \begin{array}{c c c}
                            0        & 0 \\ 
                            \varphi_{21}        & 0 
                     \end{array}
                     \right) \Big)
                     \oplus
      \Big((E/E_1)/I,0\Big),
     $$
     where $\varphi_{21}$ is induced from $\Phi$.
     The associated graded vector bundle is
     $E_0 = \Gr(E_0) = E_1 \oplus (E/E_1)/I \oplus I$ and its
     Harder--Narasimhan type is
     $\HNT(E_0)=(\mu_1,\mu,-\frac{1}{3}\mu_1+\frac{2}{3}\mu_2+\frac{2}{3}\mu_3)$.

    \item[(1.3)] The slope of $I$ satisfies 
      $-\frac{1}{3}\mu_1+\frac{2}{3}\mu_2+\frac{2}{3}\mu_3<\mu(I)\leq
      \mu_3$. Then $(E_0,\Phi_0)$ is the following Hodge bundle of type
      $(1,1,1)$:
      $$(E_0,\Phi_0) = \Big(E_1\oplus I \oplus
      (E/E_1)/I, \left( \begin{array}{c c c}
                            0        & 0 & 0\\ 
                        \varphi_{21} & 0 & 0\\
                            0        & \varphi_{32} & 0
                       \end{array}
                     \right) \Big).
     $$
     Here $\varphi_{21}$ and $\varphi_{32}$ are induced from $\Phi$.
     The associated graded vector bundle is
     $E_0 = \Gr(E_0) = E_1 \oplus (E/E_1)/I \oplus I$ and its
     Harder--Narasimhan type is
     $\HNT(E_0)=(\mu_1,\mu_2+\mu_3-\mu(I),\mu(I))$.

   \item[(1.4)] The slope of $I$ satisfies $\mu(I)=\mu_2$. Then the
     strict inequality $\mu_3<\mu_2$ holds, the line bundle
     $I=E_2/E_1$, and $(E_0,\Phi_0)$ is the following Hodge bundle of
     type $(1,1,1)$:
      $$(E_0,\Phi_0) = \Big(E_1\oplus E_2/E_1 \oplus
      E/E_2, \left( \begin{array}{c c c}
          0        & 0 & 0\\
          \phi_{21} & 0 & 0\\
          0 & \varphi_{32} & 0
                       \end{array}
                     \right) \Big),
     $$ 
     where $\varphi_{32}$ is induced from $\Phi$.
     The associated graded vector bundle is $E_0 = \Gr(E_0) = \Gr(E)$.
    \end{enumerate}
  \item[(2)] Suppose that $\mu_2>\mu$. Then $\mu_1\geq\mu_2>\mu_3$,
    the subbundle $N\subset E_2$ defined in \eqref{eq:def-N} is a
    line bundle,  and
    one of the following alternatives holds.
    \begin{enumerate}
    \item[(2.1)] The slope of $N$ satisfies $\mu_1+\mu_2-\mu_3-(2g-2)
\leq \mu(N)<\mu$. Then $(E_0,\Phi_0)$ is the following Hodge
      bundle of type $(2,1)$:
      $$
      (E_0,\Phi_0) = \Big(E_2\oplus E/E_2,
      \left( \begin{array}{c c}
                            0     & 0\\ 
                        \phi_{32} & 0
                       \end{array}
      \right) \Big).
      $$
      The associated graded vector bundle is $\Gr(E_0) = \Gr(E)$.

    \item[(2.2)] The slope of $N$ satisfies
      $\mu=\mu(N)$. Then
     $(E_0,\Phi_0)$ is the following strictly polystable Hodge bundle:
      $$
      (E_0,\Phi_0) = (N,0) \oplus \Big({E}_2/N \oplus E/E_2,
      \left( \begin{array}{c c c}
                           0 & 0\\ 
                           \varphi_{32} & 0
                       \end{array}
                     \right) \Big)
      $$
      where $\varphi_{32}$ is induced from
      $\Phi$. The associated graded vector bundle is
      $E_0 = \Gr(E_0) = E_2/N \oplus N \oplus E/E_2$ and its
      Harder--Narasimhan type is
      $\HNT(E_0) = (\frac{2}{3}\mu_1+\frac{2}{3}\mu_2-\frac{1}{3}\mu_3,\mu,\mu_3)$.

    \item[(2.3)] The slope of $N$ satisfies
      $\mu<\mu(N)\leq\mu_2$. Then $(E_0,\Phi_0)$ is the following
      Hodge bundle of type $(1,1,1)$:
      $$
      (E_0,\Phi_0) = \Big(N \oplus {E}_2/N \oplus E/E_2,
      \left( \begin{array}{c c c}
                            0        & 0 & 0\\ 
                        \varphi_{21} & 0 & 0\\
                            0        & \varphi_{32} & 0
                       \end{array}
                     \right) \Big)
      $$
      where $\varphi_{21}$ and $\varphi_{32}$ are induced from
      $\Phi$. The associated graded vector bundle is
      $E_0 = \Gr(E_0) = E_2/N \oplus N \oplus E/E_2$ and its
      Harder--Narasimhan type is
      $\HNT(E_0) = (\mu_1+\mu_2-\mu(N),\mu(N),\mu_3)$.

    \item[(2.4)] The slope of $N$ satisfies $\mu(N)=\mu_1$. Then the strict
      inequality $\mu_1>\mu_2$ holds, the line bundle $N=E_1$
      and $(E_0,\Phi_0)$ is the following Hodge bundle of type
      $(1,1,1)$:
      $$(E_0,\Phi_0) = \Big(E_1\oplus E_2/E_1 \oplus
      E/E_2, \left( \begin{array}{c c c}
          0        & 0 & 0\\
          \varphi_{21} & 0 & 0\\
          0 & \varphi_{32} & 0
                       \end{array}
                     \right) \Big),
     $$ 
     where $\varphi_{21}$ and $\varphi_{32}$ are induced from $\Phi$.
     The associated graded vector bundle is $E_0= \Gr(E_0) = \Gr(E)$.
    \end{enumerate}
\item[(3)] Suppose that $\mu_2 = \mu$. Then $\mu_1>\mu_2>\mu_3$, 
    the subbundles $I \subset E/E_1$ and $N\subset E_2$ defined in
    \eqref{eq:def-I} and \eqref{eq:def-N} are
    line bundles,  and
    one of the following alternatives holds.
    \begin{enumerate}
    \item[(3.1)]  The equivalent conditions $N=E_1$ and $I = E_2 /
      E_1$ hold. Then $(E_0,\Phi_0)$ is the following Hodge bundle of type
      $(1,1,1)$:
      $$(E_0,\Phi_0) = \Big(E_1\oplus E_2/E_1 \oplus
      E/E_2, \left( \begin{array}{c c c}
          0        & 0 & 0\\
          \varphi_{21} & 0 & 0\\
          0 & \varphi_{32} & 0
                       \end{array}
                     \right) \Big),
     $$ 
     where $\varphi_{21}$ and $\varphi_{32}$ are induced from $\Phi$.
     The associated graded vector bundle is $E_0 = \Gr(E_0) = \Gr(E)$.
    \item[(3.2)]  Otherwise $(E_0,\Phi_0)$ is the following strictly
      polystable Hodge bundle:
      $$(E_0,\Phi_0) = \Big(E_1 \oplus
      E/E_2, \left( \begin{array}{c c c}
          0 & 0\\
         \varphi_{31} & 0
         \end{array}
       \right) \Big)
       \oplus (E_2/E_1,0),
     $$ 
     where $\varphi_{31}$ is induced from $\Phi$.
     The associated graded vector bundle is $\Gr(E_0) = \Gr(E)$.
    \end{enumerate}
  \end{enumerate}
\end{Th}

\begin{Rmk}
  The Cases (1.2), (2.2) and (3) cannot happen when the rank
  and degree are co-prime, i.e., $(3,d)=1$.
\end{Rmk}

\begin{Rmk}
  \label{rem:range}
  The condition $\mu_2<\mu$ is equivalent to
  $\mu_3>-\frac{1}{3}\mu_1+\frac{2}{3}\mu_2+\frac{2}{3}\mu_3$. In
  particular the range for $\mu(I)$ in Case~(1.2) is non-empty.
\end{Rmk}

Before proceeding with the proof of Theorem~\ref{stratifications} we
deduce a couple of interesting consequences.  The theorem shows that,
in general, knowledge of the Harder--Narasimhan type of $E$ does not
suffice to determine the underlying bundle $E_0$ of the limit
$(E_0,\Phi_0) = \lim_{z\to 0}(E,z\cdot\Phi)$. However, there are some
Harder--Narasimhan types $(\mu_1,\mu_2,\mu_3)$ for which $E_0$ is
determined by $E$. We note that, by
Proposition~\ref{prop:mu-inequalities}, one has $0\leq \mu_1-\mu_3\leq
4g-4$.

\begin{Cor}
\label{cor:111a}
  Let $(E,\Phi) \in \mathcal{M}(3,d)$ be such that $E$ is an unstable
  vector bundle of slope $\mu$ and Harder--Narasimhan type
  $(\mu_1,\mu_2,\mu_3)$. Assume that $\mu_1-\mu_3>2g-2$. Then the
  limit $(E_0,\Phi_0)= \lim_{z\to0}(E,z\cdot \Phi)$ is given by
  (1.4) of Theorem~\ref{stratifications} if $\mu_2<\mu$, by
  (2.4) of Theorem~\ref{stratifications} if $\mu_2>\mu$, and by
  (3.1) of Theorem~\ref{stratifications} if $\mu_2=\mu$.
  In particular $E_0 = \Gr(E_0) = \Gr(E)$.
\end{Cor}

\begin{proof}
  We only have to show that in all the other cases of
  Theorem~\ref{stratifications} we have $\mu_1-\mu_3\leq 2g-2$.

  In Cases~(1.1), (1.2) and (1.3) we have $\mu(I)\leq \mu_3$
  (cf.~Remark~\ref{rem:range}). Moreover, by (1) of
  Proposition~\ref{prop:mu-I-N}, we have $\mu_1-(2g-2)\leq\mu(I)$. It
  follows that $\mu_1-(2g-2)\leq\mu_3$ as desired.

  Similarly, in Cases~(2.1), (2.2) and (2.3) we have $\mu(N)\leq\mu_2$ and,
  by (2) of Proposition~\ref{prop:mu-I-N}, $\mu_1+\mu_2-\mu_3-(2g-2)
  \leq \mu(N)$. Hence $\mu_1+\mu_2-\mu_3-(2g-2)\leq\mu_2$ which gives
  the conclusion.

  Finally, in Case~(3.2) we have $\varphi_{31}\neq 0$ (since otherwise
  $E$ would be semistable) and hence $\mu_1-\mu_3\leq 2g-2$.
\end{proof}

In a similar vein, we shall next see that certain types of Hodge bundles
can only be the limit of a Higgs bundle whose underlying vector bundle
has the same Harder--Narasimhan type as that of the Hodge bundle.

Before stating the result we recall (see, e.g., \cite{got} or
Hausel--Thaddeus~\cite{hausel-thaddeus:2003}) that fixed points of type
$(1,1,1)$ of the form
\begin{displaymath}
  (E_0,\Phi_0) = (L_1 \oplus L_2 \oplus L_3,
\left( \begin{array}{c c c}
               0      &       0      & 0\\
         \varphi_{21} &       0      & 0\\
               0      & \varphi_{32} & 0
        \end{array} 
 \right)
)
\end{displaymath}
are usually parametrised by the numerical invariants
\begin{align*}
  m_1 &= \deg(L_2)-\deg(L_1) + 2g-2,\\
  m_2 &= \deg(L_3)-\deg(L_2) + 2g-2,
\end{align*}
subject to the conditions
\begin{align*}
  m_i & \geq 0, \quad i=1,2,\\
  2m_1+m_2 &< 6g-6,\\
  m_1+2m_2 &< 6g-6,\\
  m_1+2m_2 &\equiv 0 \pmod 3.
\end{align*}
%(The last condition was overlooked in \cite{got}.)
For our purposes it is more natural to translate to the invariants
$(l_1,l_2,l_3)$ with $l_i=\mu(L_i)=\deg(L_i)$ (subject to
the condition $l_1+l_2+l_3=3\mu$). We then have corresponding components
$F_{(l_1,l_2,l_3)}$ of the fixed locus
% consisting of
% $(E,\Phi)$ with invariants $(l_1,l_2,l_3)$ and 
and the invariants $(l_1,l_2,l_3)$ are subject to the constraints
\begin{align*}
  l_{i+1}-l_i + 2g-2&\geq 0,\quad i=1,2,\\
  \frac{1}{3}l_1+\frac{1}{3}l_2-\frac{2}{3}l_3 &>0,\\
  \frac{2}{3}l_1-\frac{1}{3}l_2-\frac{1}{3}l_3 &>0.
\end{align*}

\begin{Cor}
\label{cor:111b}
  Let $(E_0,\Phi_0) = 
  (L_1\oplus L_2\oplus L_3,
   \left(
    \begin{smallmatrix}
       0      &       0      & 0\\
         \varphi_{21} &       0      & 0\\
               0      & \varphi_{32} & 0
    \end{smallmatrix}
   \right)
  )$ be a Hodge bundle
  of type $(1,1,1)$ with $\mu(L_1)-\mu(L_3)>2g-2$. Then
  $\mu(L_1)>\mu(L_2)>\mu(L_3)$ and any $(E,\Phi)$ such that
  $\lim_{z\to0}(E,z\cdot \Phi) = (E_0,\Phi_0)$ satisfies
  $E_0 = \Gr(E_0) = \Gr(E)$.
\end{Cor}

\begin{proof}
  It is easy to see that polystability of $(E_0,\Phi_0)$ and the
  condition $\mu(L_1)-\mu(L_3)>2g-2$ together imply that
  $\varphi_{21}$ and $\varphi_{32}$ non-zero.

  Inspecting the various cases of Theorem~\ref{stratifications} we
  see that only in cases~(1.4), (2.4) and (3.1) the limit is a
  Hodge bundle of type $(1,1,1 )$ with $\mu(L_1)-\mu(L_3)>2g-2$.
  The conclusion follows since in these cases $E_0=\Gr(E_0)=\Gr(E)$.
\end{proof}

The two previous corollaries lead to an identification between
Shatz and Bia{\l}ynicki--Birula strata in some cases. 
Recall that $U^+_{(l_1,l_2,l_3)}$ denotes the Bia{\l}ynicki-Birula
stratum of Higgs bundles whose limits lie in $F_{(l_1,l_2,l_3)}$ and
that $U'_{(l_1,l_2,l_3)}$ denotes the Shatz stratum of Higgs bundles
whose Harder--Narasimhan type is $(l_1,l_2,l_3)$.

\begin{Th}
  \label{thm:HN=BB}
  Let $(l_1,l_2,l_3)$ be such that $l_1-l_3>2g-2$. Then the
  corresponding Shatz and Bia{\l}ynicki-Birula strata in
  $\mathcal{M}(3,d)$ coincide:
  \begin{displaymath}
    U'_{(l_1,l_2,l_3)} = U^+_{(l_1,l_2,l_3)}.
  \end{displaymath}
  \qed
\end{Th}

\subsection{Proof of Theorem~\ref{stratifications}}

For the proof, we adopt the complex analytic point of view as
explained in Section~\ref{sec:higgs-bundles}. Let $\mathcal{E}$ be the
$C^\infty$ bundle underlying $E$ and consider the pair
$(\dbar_E,\Phi)$ representing $(E,\Phi)$ in the configuration space of
all Higgs bundles.  Our strategy of proof is to find a family of gauge
transformations $g(z)\in\mathcal{G}^\C$, parametrised by $z\in\C^*$, such that the limit in the
configuration space
\begin{displaymath}
  (\dbar_{E_0},\Phi_0) = \lim_{z\to 0}\bigl(g(z)\cdot(\dbar_E,z\cdot\Phi)\bigr)
\end{displaymath}
gives a \emph{stable} Higgs bundle $(E_0,\Phi_0)$. It will then follow that
$(E_0,\Phi_0)$ represents the limit in the moduli space.

We now need to consider several cases.

\subsubsection{Proof of Theorem~\ref{stratifications} -- Case (1)}
\label{sec:proof-case1}

Suppose that $\mu_2<\mu$. Then, since $\mu_1>\mu$, we must have
$\mu_1>\mu_2\geq \mu_3$. It follows from (1) of
Proposition~\ref{prop:mu-I-N} that $I \subset E/E_1$ is a line bundle
and that $\mu_1-(2g-2) \leq \mu(I) \leq \mu_2$.

We consider two separate cases.

\paragraph{Case A: $\mu_1-(2g-2) \leq \mu(I) <
-\frac{1}{3}\mu_1+\frac{2}{3}\mu_2+\frac{2}{3}\mu_3$.}
\mbox{}

We have a short exact sequence $0 \rightarrow E_1 \rightarrow E
\rightarrow E/E_1 \rightarrow 0.$ Let $\mathcal{E}$, $\mathcal{E}_1$
and $\mathcal{E}_2$ be the $C^\infty$ vector bundles underlying $E$,
$E_1$ and $E/E_1$, respectively. Then
\begin{equation}
  \label{eq:smooth-decomposition-12}
  \mathcal{E} \cong {\mathcal{E}}_1 \oplus \mathcal{E}_2
\end{equation}
and the holomorphic structure on $\mathcal{E}$ is given by the
$\dbar$-operator:
$$
\bar{\partial}_E = \left( \begin{array}{c c}
                        \bar{\partial}_1 & \beta \\ 0 & \bar{\partial}_2
                       \end{array}
\right),
$$
where $\bar{\partial}_1$ and $\bar{\partial}_2$ are $\dbar$-operators
defining the holomorphic structures on $\mathcal{E}_1$ and
$\mathcal{E}_2$, respectively, and $\beta \in
A^{0,1}(\Hom(\mathcal{E}_2,\mathcal{E}_1))$.  With respect
to the smooth decomposition (\ref{eq:smooth-decomposition-12}), the
Higgs field $\Phi\in A^{1,0}(\End(\mathcal{E}))$ takes the form:
$$
\Phi = \left( \begin{array}{c c}
                        \phi_{11} & \phi_{12}\\ 
                        \phi_{21} & \phi_{22}
                \end{array}
\right).
$$

Consider, for each $z\in \C^ *$, the constant gauge transformation
$g(z)\in\mathcal{G}^\C$ defined by
$$
g(z):=\left( \begin{array}{c c}
                        1 & 0\\ 0 & z\cdot I
                       \end{array}
 \right),
$$
with respect to the decomposition (\ref{eq:smooth-decomposition-12}).
Then:
\begin{displaymath}
  g(z)\cdot (z\cdot \Phi)=g(z)^{-1}(z\cdot \Phi)g(z)
  =\left( \begin{array}{c c} z \cdot \phi_{11} & z^2 \cdot \phi_{12}\\
      \phi_{21} & z \cdot \phi_{22}
    \end{array}
  \right) 
  \to \left( \begin{array}{c c} 0 & 0\\ \phi_{21} & 0
    \end{array}
  \right)\ \textmd{when}\ z\to 0
\end{displaymath}
and, moreover,
$$
g(z)\cdot \bar{\partial}_{E}=g(z)^{-1}\circ\bar{\partial}_{E}\circ g(z)= \left( \begin{array}{c c}
    \bar{\partial}_1 & z\cdot \beta \\ 0 & \bar{\partial}_2
                       \end{array}
 \right)
 \to 
 \left( \begin{array}{c c}
                        \bar{\partial}_1 & 0 \\ 0 & \bar{\partial}_2
                       \end{array}
 \right)\
 \textmd{when}\ z\to0.
$$
Note that this simple formula for the gauge transformed
$\bar{\partial}$-operator is valid because the gauge transformation is
constant on $X$.
Thus, in the configuration space of all Higgs bundles the limit
$\lim_{z\to 0}z\cdot(E,\Phi)$ is gauge equivalent to
\begin{displaymath}
  (E_0,\Phi_0) = \Bigl(E_1\oplus E/E_1,
  \begin{pmatrix}
    0 & 0\\
    \phi_{21} & 0
  \end{pmatrix}
\Bigr).
\end{displaymath}
This Higgs bundle will represent the limit in the moduli space
$\mathcal{M}(3,d)$ provided that it is stable.

To show stability, we note that there are three kinds of
$\Phi_0$-invariant subbundles of $E_0$, namely $E_1\oplus I$, $E/E_1$, and
an arbitrary line bundle $L \subset E/E_1$. We deal with each case
in turn:

\begin{enumerate}
\item \emph{The subbundle $E_1\oplus I \subset E_1\oplus E/E_1$.}  By
  hypothesis
  $\mu(I)<-\frac{1}{3}\mu_1+\frac{2}{3}\mu_2+\frac{2}{3}\mu_3$ which
  is equivalent to $\mu(E_1\oplus I)<\mu(E)=\mu(E_0)$ as required.
\item 
  \emph{The subbundle $E/E_1\subset E_1\oplus E/E_1$.}
  It is immediate from the properties of the Harder--Narasimhan
  filtration that $\mu(E/E_1)<\mu(E)=\mu(E_0)$.
\item \emph{A line subbundle $L \subset E/E_1$.} From the properties
  of the Harder--Narasimhan filtration we have that either $E_2/E_1
  \subset E/E_1$ is maximal destabilizing (if $\mu_2<\mu_3$) or $E/E_1$ is
  semistable (if $\mu_2=\mu_3$). Either way we have that
  $\mu(L)\leq\mu_2$. Since $\mu_2<\mu=\mu(E)$ by hypothesis, it
  follows that $\mu(L)<\mu(E)=\mu(E_0)$.
\end{enumerate}

Finally note that, clearly, $\Gr(E_0) = E_1\oplus E_2/E_1\oplus E/E_2
=\Gr(E)$. Altogether we have seen that, under the given conditions on
the slope of $I$, the limiting bundle $(E_0,\Phi_0)$ is as stated in
Case~(1.1) of the theorem.

\paragraph{Case B:
  $-\frac{1}{3}\mu_1+\frac{2}{3}\mu_2+\frac{2}{3}\mu_3\leq\mu(I)\leq \mu_2$.}
\mbox{}

Define $Q=(E/E_1)/I$ so that we have a short exact sequence $0
\rightarrow I \rightarrow E/E_1 \rightarrow Q \rightarrow 0.$ Let
$\mathcal{E}_1$, $\mathcal{I}$ and $\mathcal{Q}$ be the $C^\infty$
bundles underlying $E_1$, $I$ and $Q$, respectively, so that we have a
$C^\infty$-decomposition
\begin{equation}
  \label{eq:smooth-decomposition111}
  \mathcal{E} = \mathcal{E}_1 \oplus \mathcal{I} \oplus \mathcal{Q}.
\end{equation}
Recalling that $\mathcal{I}$ comes from $\Phi(E_1)\otimes K^{-1}$, we
may write the Higgs field $\Phi$ as:
$$
\Phi = 
 \left( \begin{array}{c c c}
                        \varphi_{11} & \varphi_{12} & \varphi_{13}\\ 
                        \varphi_{21} & \varphi_{22} & \varphi_{23}\\
                              0      & \varphi_{32} & \varphi_{33}
                       \end{array}
 \right)
$$
with respect to the decomposition \eqref{eq:smooth-decomposition111}.
Moreover, the holomorphic structure on $E$ is of the form
$$
\displaystyle \bar{\partial}_E = \left( \begin{array}{c c c}
                        \bar{\partial}_1 &     \beta_{12}   & \beta_{13}\\ 
                                0        & \bar{\partial}_2 & \beta_{23}\\ 
                                0        &         0        & \bar{\partial}_3  
                       \end{array}
 \right).
$$
Now, for each $z\in\C^*$ take the following constant gauge
transformation:
$$g(z):=\left( \begin{array}{c c c}
                        1 & 0 & 0\\ 
                        0 & z & 0\\
                        0 & 0 & z^2
               \end{array}
             \right)
$$
of $\mathcal{E}$ with respect to the decomposition
\eqref{eq:smooth-decomposition111}.  Then
\begin{displaymath}
  \begin{split}
    g(z)\cdot (z\cdot \Phi)&=g(z)^{-1}(z\cdot \Phi)g(z) \\
    &=
    \left( \begin{array}{c c c}
             z\cdot \varphi_{11} & z^2\cdot \varphi_{12} & z^3\cdot \varphi_{13}\\
             \varphi_{21}       & z  \cdot \varphi_{22} & z^2\cdot \varphi_{23}\\
             0 & \varphi_{32} & z \cdot \varphi_{33}
           \end{array}
         \right) \longrightarrow \left( \begin{array}{c c c}
                                          0      &       0      & 0\\
                                          \varphi_{21} &       0      & 0\\
                                          0 & \varphi_{32} & 0
                                        \end{array} 
                                      \right)\ \text{when}\ z\to0                                
\end{split}
\end{displaymath}
and
\begin{displaymath}
  \begin{split}
    g(z)\cdot \bar{\partial}_{E} &=
    g(z)^{-1}\circ\bar{\partial}_{E}\circ g(z) \\
    &=
    \left( \begin{array}{c c c}
             \bar{\partial}_1 & z \cdot \beta_{12}   & z^2\cdot \beta_{13}\\
             0        &  \bar{\partial}_2    &  z \cdot \beta_{23}\\
             0 & 0 & \bar{\partial}_3
           \end{array}
         \right) \longrightarrow \left( \begin{array}{c c c}
                                          \bar{\partial}_1 &         0        & 0\\
                                          0        & \bar{\partial}_2 & 0\\
                                          0 & 0 & \bar{\partial}_3
                                        \end{array}
                                      \right)\ \text{when}\ z \to0.
                                    \end{split}
\end{displaymath}
Hence, in the configuration space, $\lim_{z\to 0}z\cdot(E,\Phi)$ is
gauge equivalent to
\begin{displaymath}
(E_0,\Phi_0) = \Big(E_1\oplus I \oplus
      (E/E_1)/I, \left( \begin{array}{c c c}
                            0        & 0 & 0\\ 
                        \varphi_{21} & 0 & 0\\
                            0        & \varphi_{32} & 0
                       \end{array}
                     \right) \Big).
\end{displaymath}
Now we prove that $(E_0,\Phi_0)$ is a semistable Higgs bundle.
The $\Phi_0$-invariant subbundles of $E_0$ are the following:
\begin{enumerate}
\item \emph{The subbundle $I \oplus (E/E_1)/I \subset E_0$}. We have
  that $\mu(I \oplus (E/E_1)/I )<\mu(E) \iff \mu(E_1)>\mu(E)$, which
  holds by properties of the Harder--Narasimhan filtration.
\item \emph{The subbundle $(E/E_1)/I \subset E_0$}. The condition
  $\mu((E/E_1)/I)\leq\mu(E)$ is equivalent to
  $-\frac{1}{3}\mu_1+\frac{2}{3}\mu_2+\frac{2}{3}\mu_3\leq\mu(I)$ which
  holds by assumption.
\end{enumerate}

Consider the situation when
$-\frac{1}{3}\mu_1+\frac{2}{3}\mu_2+\frac{2}{3}\mu_3=\mu(I)$; this is
the only case in which
$(E_0,\Phi_0)$ is strictly semistable. Then $Q=(E/E_1)/I$ is a
$\Phi$-invariant subbundle with $\mu(Q) = \mu$, and it follows that the
polystable representative of the $S$-equivalence class of $(E_0,\Phi_0)$
is obtained by setting $\varphi_{32}=0$ in $\Phi_0$.  This leads
to the description given in Case~(1.2).

It remains to analyze the Harder--Narasimhan type of $E_0$ when
$-\frac{1}{3}\mu_1+\frac{2}{3}\mu_2+\frac{2}{3}\mu_3\neq\mu(I)$.
There are two situations to consider.

The first situation is when $\mu(I) \leq \mu(Q)$. Then
the Harder--Narasimhan type of $E_0$ is
$\HNT(E_0)=(\mu(E_1),\mu(Q),\mu(I))$.  Hence, using Shatz's theorem
\cite[Theorem~3]{sha} that the Harder--Narasimhan polygon rises under
specialization, we conclude that $\mu(I) \leq \mu(E/E_2)$. This leads
to the description given in Case~(1.3).

The second situation is when $\mu(I) > \mu(Q)$. Then the
Harder--Narasimhan type of $E_0$ is
$\HNT(E_0)=(\mu(E_1),\mu(I),\mu(Q))$. Hence, from Shatz's theorem we
deduce that $\mu(I) \geq \mu(E_2/E_1)$. But $I \subset E/E_1$ so, from
the properties of the Harder--Narasimhan filtration, we conclude that
in fact $\mu(I)=\mu_2$. If $\mu_3=\mu_2$ it follows that
$\mu(I)=\mu(Q)$, contradicting $\mu(I) > \mu(Q)$. Hence
$\mu_3<\mu_2$ and $I \subset E/E_1$ is the unique maximal
destabilizing subbundle, i.e., $I=E_2/E_1$ and so Case~(1.4)
occurs.

This completes the proof of Case~(1).

\begin{Rmk}
  \label{rem:rank2-proof}
  The arguments given for Case A above apply word for word to prove
  Proposition \ref{prop:limit-rank2}, except that the argument to show
  that $(E_0,\Phi_0)$ is a semistable Higgs bundle is simpler: indeed,
  in the rank 2 case, the only $\Phi$-invariant subbundle of $E_0$ is
  $E/E_1$. This satisfies $\mu(E/E_1) < \mu(E) = \mu(E_0)$ because the
  subbundle $E_1$ is destabilizing, i.e., $\mu(E_1)>\mu(E)$.
\end{Rmk}

\subsubsection{Proof of Theorem~\ref{stratifications} -- Case (2)}
\label{sec:strat-21}

Suppose that $\mu_2>\mu$. Then, since $\mu_3<\mu$, we must have $\mu_1
\geq \mu_2 > \mu_3$. It follows from (2) of
Proposition~\ref{prop:mu-I-N} that $N \subset E_2$ is a line bundle
and that $\mu_1+\mu_2-\mu_3-(2g-2) \leq \mu(N) \leq \mu_1$.

We consider two separate cases.

\paragraph{Case C: $\mu_1+\mu_2-\mu_3-(2g-2) \leq \mu(N)<\mu$.}
\mbox{}

We have a short exact sequence $0 \rightarrow E_2 \rightarrow E
\rightarrow E/E_2 \rightarrow 0.$ Let $\mathcal{E}$, $\mathcal{E}_2$
and $\mathcal{E}_3$ be the $C^\infty$ vector bundles underlying $E$,
$E_2$ and $E/E_2$, respectively.
Then
\begin{math}
  \mathcal{E} \cong {\mathcal{E}}_2 \oplus \mathcal{E}_3
\end{math}
and the holomorphic structure on $\mathcal{E}$ is given by a
$\dbar$-operator of the form
$
\bar{\partial}_E = \left( 
  \begin{smallmatrix}
    \bar{\partial}_2 & \beta \\
    0 & \bar{\partial}_3
  \end{smallmatrix}
\right)
$, while the
Higgs field $\Phi\in A^{1,0}(\End(\mathcal{E}))$ takes the form:
$
\Phi = \left( 
  \begin{smallmatrix}
    \phi_{22} & \phi_{23}\\
    \phi_{32} & \phi_{33}
  \end{smallmatrix}
\right) $.  The same calculation as in Case~A shows that in the
configuration space of all Higgs bundles, $\lim_{z\to
  0}z\cdot(E,\Phi)$ is gauge equivalent to
\begin{displaymath}
  (E_0,\Phi_0) = \Bigl(E_2\oplus\ E/E_2,
  \begin{pmatrix}
    0 & 0\\
    \phi_{32} & 0
  \end{pmatrix}
\Bigr).
\end{displaymath}
This Higgs bundle will represent the limit in the moduli space
$\mathcal{M}(3,d)$ if it is stable. There are three kinds of
$\Phi_0$-invariant subbundles to check:

\begin{enumerate}
\item \emph{The subbundle $N \subset E_2\oplus E/E_2$.}  By
  hypothesis $\mu(N)<\mu=\mu(E)=\mu(E_0)$.
\item 
  \emph{The subbundle $E/E_2\subset E_2\oplus E/E_2$.}
  It is immediate from the properties of the Harder--Narasimhan
  filtration that $\mu(E/E_2)<\mu(E)=\mu(E_0)$.
\item \emph{Subbundles $L\oplus E/E_2 \subset E_2\oplus E/E_2$ for $L
    \subset E_2$ a line subbundle.}  From the properties of the
  Harder--Narasimhan filtration we have that either $E_1 \subset
  E_2$ is maximal destabilizing (if $\mu_1>\mu_2$) or $E_2$ is
  semistable (if $\mu_1=\mu_2$). Either way we have that
  $\mu(L)\leq\mu_1$. It
  follows that
  \begin{align*}
    2\mu(L\oplus E/E_2) &=\mu(L) + 3\mu -\mu_1-\mu_2 \\
    &\leq 3\mu-\mu_2 \\
    &<2\mu,
  \end{align*}
  where we have used the hypothesis $\mu_2>\mu$ in the last
  step. Hence $\mu(L\oplus E/E_2)<\mu=\mu(E)=\mu(E_0)$ as desired.
\end{enumerate}

Finally note that, clearly, $\Gr(E_0) = E_1\oplus E_2/E_1\oplus E/E_2
=\Gr(E)$. Altogether we have seen that, under the given conditions on
the slope of $I$, the limiting bundle $(E_0,\Phi_0)$ is as stated in
Case~(2.1) of the theorem.

\paragraph{Case D: $\mu\leq \mu(N)\leq\mu_1$.}
\mbox{}

Define $R=E_2/ N$ so that we have a short exact sequence
$0 \rightarrow N \rightarrow E_2 \rightarrow R \rightarrow 0.$ Let
$\mathcal{N}$, $\mathcal{R}$ and $\mathcal{E}_3$ be the $C^\infty$
bundles underlying $N$, $R$ and $E/E_2$, respectively, so that
we have a decomposition of $C^\infty$-bundles 
\begin{equation}
  \label{eq:smooth-decomposition111b}
  \mathcal{E} = \mathcal{N} \oplus \mathcal{R} \oplus \mathcal{E}_3.
\end{equation}
Recalling that $\mathcal{N}$ comes from $\ker(\phi_{21})$, we may write the Higgs field 
$\Phi$ as:
$$\Phi = 
 \left( \begin{array}{c c c}
                        \varphi_{11} & \varphi_{12} & \varphi_{13}\\ 
                        \varphi_{21} & \varphi_{22} & \varphi_{23}\\
                              0      & \varphi_{32} & \varphi_{33}
                       \end{array}
 \right)$$
with respect to the decomposition \eqref{eq:smooth-decomposition111b}.
Moreover, the holomorphic structure on $E$ is of the form
$$\displaystyle \bar{\partial}_E = \left( \begin{array}{c c c}
                        \bar{\partial}_1 &     \beta_{12}   & \beta_{13}\\ 
                                0        & \bar{\partial}_2 & \beta_{23}\\ 
                                0        &         0        & \bar{\partial}_3  
                       \end{array}
 \right).$$
Now take the constant gauge transformation
$g(z)=\left( \begin{smallmatrix}
                        1 & 0 & 0\\ 
                        0 & z & 0\\
                        0 & 0 & z^2
               \end{smallmatrix}
             \right)
$
of $\mathcal{E}$ with respect to the decomposition \eqref{eq:smooth-decomposition111b}.     
The same calculation as in Case~B shows that in the configuration space  $\lim_{z\to 0}z\cdot(E,\Phi)$ is gauge equivalent to
\begin{displaymath}
(E_0,\Phi_0) = \Big(N \oplus E_2/N \oplus E/E_2, 
\left( \begin{array}{c c c}
                            0        & 0 & 0\\ 
                        \varphi_{21} & 0 & 0\\
                            0        & \varphi_{32} & 0
                       \end{array}
                     \right) \Big).
\end{displaymath}
We now prove that $(E_0,\Phi_0)$ is a semistable Higgs bundle.  The
$\Phi_0$-invariant subbundles of $E_0$ are the following:
\begin{enumerate}
\item \emph{The subbundle $E/E_2\subset E_0$}. From the properties of the Harder--Narasimhan filtration we have  $\mu(E/E_2)<\mu(E)=\mu(E_0)$.
\item \emph{The subbundle $E_2/N \oplus E/E_2 \subset E_0$}. The
  hypothesis $\mu(N)\geq\mu$ is equivalent to $\mu(E_2/N \oplus E/E_2)\leq\mu=\mu(E)=\mu(E_0)$.
\end{enumerate}

Consider the situation when $\mu(N) = \mu$; this is the only case in
which $(E_0,\Phi_0)$ is strictly semistable. Then
$E_2/N \oplus E/E_2 \subset E_0$ is a $\Phi$-invariant subbundle of
slope $\mu(E_2/N \oplus E/E_2 \subset E_0) = \mu$, and it follows that
the polystable representative of the $S$-equivalence class of
$(E_0,\Phi_0)$ is obtained by setting $\varphi_{21}=0$ in $\Phi_0$.
This leads to the description given in Case~(2.2).

It remains to analyze the Harder--Narasimhan type of $E_0$ when
$\mu(N) \neq \mu$. For brevity we write $R=E_2/N$. There are two
situations to consider.

The first situation is when $\mu(N) \leq \mu(R)$. Then the
Harder--Narasimhan type of $E_0$ is
$\HNT(E_0)=(\mu(R),\mu(N),\mu_3)$.  Hence, once again using Shatz's
theorem, we conclude that $\mu(N) \leq \mu_2$. This leads to the description given in Case~(2.3).

The second situation is when $\mu(N) > \mu(R)$. Then the
Harder--Narasimhan type of $E_0$ is
$\HNT(E_0)=(\mu(N),\mu(R),\mu_3)$. Hence, from Shatz's theorem we
deduce that $\mu(N) \geq \mu_1$. But $N \subset E_2$ so, from
the properties of the Harder--Narasimhan filtration, we conclude that
in fact $\mu(N)=\mu_1$. If $\mu_2=\mu_1$ it follows that
$\mu(N)=\mu(R)$, contradicting $\mu(N) > \mu(R)$. Hence
$\mu_2<\mu_1$ and so $N \subset E_2$ is the unique maximal
destabilizing subbundle, i.e., $N=E_1$ and Case~(2.4)
occurs.

\subsubsection{Proof of Theorem~\ref{stratifications} -- Case (3)}
\label{sec:strat-3}

Suppose that $\mu_2 = \mu$. Then, since $E$ is unstable, we must have
$\mu_1 > \mu_2 > \mu_3$. It follows from Proposition~\ref{prop:mu-I-N}
that the subbundles $I \subset E/E_1$ and $N\subset E_2$ are line bundles.

Consider the line bundle $N \subset E_2$. If $N\neq E_1$ we have a
non-zero map
\begin{displaymath}
  N \to E_2/E_1.
\end{displaymath}
It follows that $\mu(N)\leq \mu(E_2/E_1) = \mu_2 = \mu$. Arguing as in
Case~C above, we see that in the
configuration space of all Higgs bundles, $\lim_{z\to
  0}z\cdot(E,\Phi)$ is gauge equivalent to
\begin{displaymath}
  (E_0,\Phi_0) = \Bigl(E_2\oplus\ E/E_2,
  \begin{pmatrix}
    0 & 0\\
    \phi_{32} & 0
  \end{pmatrix}
\Bigr)
\end{displaymath}
and that this is strictly semistable. Moreover, the subbundle
$E_1\oplus E/E_2$ is $\Phi$-invariant and has slope $\mu(E_1\oplus
E/E_2) = \mu_2 = \mu$. Hence the polystable representative of the
$S$-equivalence class of $(E_0,\Phi_0)$ is as stated in Case~(3.2).

Now suppose that $N = E_1$. In this case we can argue as in Case~D
above and see that the limit is as stated in  Case~(3.1).

In an analogous manner, we see that if $I\neq E_2/E_1$ the polystable
representative of  the
$S$-equivalence class of $(E_0,\Phi_0)$ is as stated in Case~(3.2),
while if $I = E_2/E_1$ the limit is as stated in  Case~(3.1).

Since the Cases (3.1) and (3.2) are mutually exclusive, we see that in
fact the conditions $N = E_1$ and $I = E_2/E_1$ are equivalent. This completes the proof of Case~(3) and thus the proof of Theorem~\ref{stratifications}.

% \bibliography{bib-data}
% \bibliographystyle{amsplain}

\end{document}